\documentclass[10pt,reno]{amsart}
\usepackage{amssymb}
\usepackage{amsmath}
\usepackage{enumitem}
\usepackage{mathrsfs}
\usepackage[english]{babel}
\usepackage[all]{xy}
\usepackage{hyperref}
\usepackage{marginnote}

\usepackage{stmaryrd}

\usepackage{fullpage}

\RequirePackage{mathrsfs} %\let\mathcal\mathcal

\newtheorem{theorem}{Theorem}[section]
\newtheorem{lemma}[theorem]{Lemma}
\newtheorem{proposition}[theorem]{Proposition}

\theoremstyle{definition}
\newtheorem*{ack}{Acknowledgements}
\newtheorem*{con}{Conventions}
\newtheorem{remark}[theorem]{Remark}
\newtheorem{example}[theorem]{Example}
\newtheorem{definition}[theorem]{Definition}

\numberwithin{equation}{section} \numberwithin{figure}{section}

\DeclareMathOperator{\Aut}{Aut}

\DeclareMathOperator{\Spec}{Spec}

\DeclareMathOperator{\an}{an}

\DeclareMathOperator{\Hom}{Hom}

\newcommand{\PGL}{\textrm{PGL}}

\newcommand{\Qbar}{\overline{\QQ}}

\newcommand\ZZ{\mathbb{Z}}

\newcommand\QQ{\mathbb{Q}}

\newcommand\CC{\mathbb{C}}

\newcommand\OO{\mathcal{O}}

\usepackage{color}

\definecolor{orange}{rgb}{1,0.5,0}

\title[Integral points on period domains]{Integral points on algebraic subvarieties of period domains: from number fields to finitely generated fields}

\author{Ariyan Javanpeykar}
\address{Ariyan Javanpeykar \\
Institut f\"{u}r Mathematik\\
Johannes Gutenberg-Universit\"{a}t Mainz\\
Staudingerweg 9, 55099 Mainz\\
Germany.}
\email{peykar@uni-mainz.de}

\author{Daniel Litt}
\address{Daniel Litt\\
Institute for Advanced Study\\
1 Einstein Drive,\\
Princeton, NJ 08540}
\email{dlitt@math.ias.edu}

 %hyperbolicity


\begin{document}

\maketitle
% \tableofcontents

\thispagestyle{empty}
 
 \begin{abstract} We show that for a variety which admits a quasi-finite period map, finiteness (resp.~non-Zariski-density) of $S$-integral points implies finiteness (resp.~non-Zariski-density) of points over all $\mathbb{Z}$-finitely generated integral domains of characteristic zero.  Our proofs rely on foundational results in Hodge theory due to Deligne, Griffiths, and Schmid, and Bakker-Brunebarbe-Tsimerman.  We give straightforward applications to Shimura varieties, locally symmetric varieties,   the moduli space of smooth hypersurfaces in projective space, and the moduli of smooth divisors in an abelian variety.
\end{abstract}

 \section{Introduction}

 The goal of this note is to give arithmetic applications of foundational results in Hodge theory. Our main abstract result (which is a combination of Theorems \ref{thm:intro} and \ref{thm:criterion_nd} below) is:
 
 \begin{theorem}[Main Result, I] \label{thm:main-theorem-intro}
Let $A\subset k = \overline{\mathbb{Q}}$ be a finitely generated subring and let $\mathcal{X}$ be a finite type $A$-scheme such that $\mathcal{X}_k$ is a quasi-projective variety  over $k$ which admits a quasi-finite complex-analytic period map. Then the following statements are equivalent.
\begin{enumerate}
\item For every finitely generated subring $A'\subset k$ containing $A$, the set $\mathcal{X}(A')$ is finite (resp.~not Zariski-dense in $\mathcal{X}(k)$).
\item For every finitely generated integral domain $B$ containing  $A$, the set $\mathcal{X}(B)$ is finite (resp.~not Zariski-dense in $\mathcal{X}(\overline{\text{Frac}(B)})$) (where $\overline{\text{Frac}(B)}$ is a choice of algebraic closure of $\text{Frac}(B)$).
\end{enumerate}
\end{theorem}
 
 In other words, for varieties admitting a quasi-finite period map, finiteness of $\mathscr{O}_{K,S}$-points (where $K$ ranges over all number fields and $S$ ranges over all finite collections of finite places of $K$) implies finiteness of $A$-points for all $\mathbb{Z}$-finitely generated integral domains $A$ of characteristic zero, and a similar statement (which requires substantially deeper input) holds for non-Zariski-density of rational points. Both the finiteness and non-density results require input from Hodge theory.  Arguably, the novel technical result in our proof of Theorem \ref{thm:main-theorem-intro} is  Theorem \ref{thm:criterion_nd}.
 
 We proceed to give various applications of these results, by applying them to various moduli spaces/stacks. For example, our first result extends Lawrence-Sawin's finiteness result for smooth hypersurfaces in an abelian variety from number fields \cite{LS} to  finitely generated fields of characteristic zero.
 
 \begin{theorem}[Main Result, II, Lawrence-Sawin + $\epsilon$]\label{thm:ls}
 Let $K$ be a number field, let $S$ be a finite set of finite places of $K$, and let $\mathcal{A}$ be an abelian scheme over $\mathcal{O}_{K,S}$. Let $D$ be an ample divisor on $\mathcal{A}_K$. Then, for any $\mathcal{O}_{K,S}$-finitely generated normal integral domain $R$ of characteristic zero, the set of $R$-smooth hypersurfaces $H\subset \mathcal{A}_R$ such that $H$ represents $D_{\mathrm{Frac}(R)}$ on $\mathcal{A}_{\mathrm{Frac}(R)}$ is finite.
 \end{theorem}
 
  We stress that the bulk of the work to prove Theorem \ref{thm:ls} is contained in the paper of Lawrence-Sawin \cite{LS}. We only \emph{combine} their work with ours to prove more general finiteness statements over finitely generated fields of characteristic zero (as opposed to only number fields).
 
Note that Theorem \ref{thm:ls} shows that the Shafarevich conjecture for smooth hypersurfaces in an abelian variety over a number field (as proven by Lawrence-Sawin) persists over finitely generated fields of characteristic zero. Also, note that the proof of Theorem \ref{thm:ls} is a straightforward consequence of our main abstract result (Theorem \ref{thm:main-theorem-intro}) and the following two facts due to Lawrence-Sawin: 
\begin{enumerate}
\item  Lawrence-Sawin's main theorem \cite[Theorem~1.1]{LS} which says that Theorem \ref{thm:ls} holds whenever $\dim R =1$;
\item The moduli space of smooth hypersurfaces in an abelian variety admits a quasi-finite period map; see \cite[Proposition~5.10]{LS};
 \end{enumerate}

Our next result is of similar nature, but is instead concerned with the Shafarevich conjecture for smooth hypersurfaces in projective space (as opposed to smooth hypersurfaces in a fixed abelian variety).

\begin{theorem}[Main Result, III]\label{thm:hypsurf_intro} Let $d\geq 3$ be an integer and let $n\geq 2$. Assume that, for every number field $K$ and every finite set of finite places $S$ of $K$, the set of $\OO_{K,S}$-isomorphism classes of smooth hypersurfaces of degree $d$ in $\mathbb{P}^{n+1}_{\OO_{K,S}}$ is finite. Then, for every $\ZZ$-finitely generated normal integral domain $A$ of characteristic zero, the set of $A$-isomorphism classes of smooth hypersurfaces of degree $d$ in $\mathbb{P}^{n+1}_{A}$ is finite.
\end{theorem}

Note that Theorem \ref{thm:hypsurf_intro} says that the Shafarevich conjecture for smooth hypersurfaces over number fields implies the analogous conjecture for smooth hypersurfaces over all finitely generated fields of characteristic zero; we refer the reader to \cite{JL, JLFano, JLM} for related  results on this conjecture.

 We now briefly discuss our main abstract result (Theorem \ref{thm:main-theorem-intro}); we will discuss the finiteness and non-Zariski-density statements separately, as the proofs are somewhat different. 
 
  \subsection{Finiteness results}
 
 For convenience, we rephrase the part of Theorem \ref{thm:main-theorem-intro} about finiteness in terms of the notion of arithmetic hyperbolicity. Lang introduced the notion of arithmetic hyperbolicity over $\Qbar$ (sometimes also referred to as \emph{Mordellicity}) to  appropriately formalize the property of ``having only finitely many rational points''.  

\begin{definition}[Arithmetic hyperbolicity]\label{defn:arhyp}  Let $k$ be an algebraically closed field of characteristic zero. A finite type separated scheme 
  $X$  over $k$ is \emph{arithmetically hyperbolic over $k$} 
if  there is a  $\ZZ$-finitely generated subring $A\subset k$, a finite type separated $A$-scheme $\mathcal{X}$ and an isomorphism  of schemes $\mathcal{X}_k \cong X$ over $k$ such that, for all $\ZZ$-finitely generated subrings $ A'\subset k$ containing $A$,  the set $\mathcal{X}(A')$  of $A'$-points  on $\mathcal{X}$ is finite.
\end{definition} 

 For example, by Faltings's finiteness theorem \cite{FaltingsComplements}, a smooth quasi-projective connected curve $X$ over $k$ is arithmetically hyperbolic over $k$  if and only if $X$ is not isomorphic to $\mathbb{P}^1_k, \mathbb{A}^1_k, \mathbb{A}^1_k\setminus \{0\}$, nor a smooth proper connected genus one curve over $k$. Faltings also proved that a closed subvariety $X$ of an abelian variety  $A$ over $k$ is arithmetically hyperbolic over $k$ if and only if $X$ does not contain the translate of a positive-dimensional abelian subvariety of $A$; see \cite{FaltingsLang}.  %The notion of arithmetic hyperbolicity is further studied in \cite{Autissier1, Autissier2, JLalg, JAut, VojtaLangExc}, and also      \cite{BauerStoll,  CLZ, Dimitrov,   Faltings2, FaltingsComplements,  Levin, Moriwaki, UllmoShimura, Vojta1, Vojta1a, Vojta2b,  Vojta3, Vojta87, VojtaSubb}.   \\
  
A period domain (usually denoted by $D$) is a classifying space for polarized Hodge structures of some fixed type.  
\begin{definition}\label{defn:periodmap}
 We say that a variety  $X$   over $k$ \emph{admits a quasi-finite complex-analytic period map (up to Galois conjugation)} if there  exists a subfield $k_0\subset k$, an embedding $k_0\to \CC$, a  variety $X_{0}$ over $k_0$, an isomorphism of $k$-schemes $X_{0,k}\cong X$, a period domain $D$, a discrete arithmetic subgroup $\Gamma$ of $\mathrm{Aut}(D)$, and a horizontal locally liftable holomorphic map $X_{0,\CC}^{\an}\to\Gamma\backslash D $ with finite fibres.
 
 \end{definition} 
  We will follow \cite{Schmid} and recall some basics of the theory in Section \ref{section:hodge_theory}.

The part of Theorem \ref{thm:main-theorem-intro} about finiteness may be rephrased (and slightly generalized) as:
\begin{theorem}[Main Result, IV] \label{thm:intro} Let $k\subset L$ be an extension of algebraically closed fields of characteristic zero.
Let $X$ be a   variety over $k$ such that $X$ admits a quasi-finite  complex-analytic period map. If $X$ is arithmetically hyperbolic  over $k$, then $X_L$ is arithmetically hyperbolic over $L$.
\end{theorem}

We note that Lang-Vojta's conjecture on integral points of varieties (see  \cite[Conj.~4.3]{Vojta3}) implies  that a variety $X$ over $\Qbar$ which admits a quasi-finite complex-analytic period map is in fact arithmetically hyperbolic over $\Qbar$, as all its subvarieties are of log-general type by a theorem of Kang Zuo \cite{ZuoNeg}.

Theorem \ref{thm:intro} can be applied to curves of genus at least two, as such curves admit a quasi-finite period map up to a finite \'etale cover \cite{MartinKodaira}. However, in this case, Faltings already remarked that the statement follows from  Grauert-Manin's finiteness theorem (\emph{formerly} the function field analogue of Mordell's conjecture); see \cite[\S VI.4,~p.215]{FaltingsComplements}. Similarly, if $g>0$ is an integer and $X$ is the moduli space of principally polarized abelian varieties of dimension $g$ over $\Qbar$ with level $3$ structure, then Faltings   showed that $X_k$ is arithmetically hyperbolic over $k$ by ``re-doing'' part of his proof that $X$ is arithmetically hyperbolic over $\Qbar$; the fact that $X$ is arithmetically hyperbolic over $\Qbar$ is precisely Shafarevich's arithmetic finiteness conjecture for principally polarized abelian schemes over   $\ZZ$-finitely generated subrings of $\Qbar$.  Around the same time, in Szpiro's seminar \cite{Szpiroa}, Martin-Deschamps gave a \emph{different} proof of the arithmetic hyperbolicity of $X_k$ by using a specialization argument on the moduli stack of principally polarized abelian schemes; see \cite{Martin}. We stress that our proof of Theorem \ref{thm:intro} is very close to Martin-Deschamps's line of reasoning. Indeed, Martin-Deschamps' proof crucially relies on   Faltings's function field analogue of the Shafarevich conjecture for abelian varieties \cite{FaltingsArakelov} and   Grothendieck's theorem on monodromy representations of abelian schemes \cite{GrothendieckHom}.     In our proof of Theorem \ref{thm:intro},  we replace these results of Faltings and Grothendieck  by foundational  results of Deligne, Griffiths, and Schmid in Hodge theory.

In fact, our proof of   Theorem \ref{thm:intro}  relies on the following consequence of Deligne's finiteness theorem for monodromy representations  \cite{Delignemonodromy} and  the ``Rigidity Theorem'' in Hodge theory (see  Theorem \ref{rigidity-theorem}).
\begin{theorem} [Deligne + Rigidity Theorem]\label{thm:deligne_intro}
Let $X$ be a variety   over $k$ which admits a quasi-finite period map (up to Galois conjugation). Then, for every variety $Y$ over $k$, every $y$ in $Y(k)$, and every $x$ in $X(k)$, the set of morphisms $f:Y\to X$ with $f(y)=x$ is finite.
\end{theorem}
 
 Note that Theorem \ref{thm:deligne_intro} is a finiteness statement about maps of pointed varieties to a period domain. It is crucial that we consider \emph{pointed} maps here; see Remark \ref{remark:ds} for a discussion of this.
 
Besides  applying our   results for varieties with a quasi-finite period to the moduli space of smooth hypersurfaces, we also give further applications   to locally symmetric varieties and  Shimura varieties.    
 
\subsection{Non-density results} It is also natural to study the non-Zariski-density  (as opposed to the finiteness) of integral points on certain moduli spaces.  Our main novel result on non-density   is Theorem \ref{thm:criterion_nd}.

We can apply Theorem \ref{thm:criterion_nd} to give conditional results on the non-density of integral points on the Hilbert scheme $\mathrm{Hilb}_{d,n}$ of smooth hypersurfaces of degree $d$ in $\mathbb{P}^{n+1}$ over $\ZZ$.

By Vojta's extension of Lang's conjecture on non-density of integral points \cite[Conj.~4.3]{Vojta3}, the finiteness of integral points on a variety is conjecturally closely related to its subvarieties being of log-general type.   However, as the Hilbert scheme of smooth hypersurfaces has subvarieties which are  not of log-general type, it is not reasonable to expect finiteness of integral points on this moduli space (and it is not hard to see that   $\mathrm{Hilb}_{d,n}$   has infinitely many $\ZZ[1/d]$-points). Nonetheless,  it follows from \cite{Javan3} that there is a finite \'etale cover $H'\to \mathrm{Hilb}_{d,n}$ such that $H'$ dominates  a positive-dimensional variety of log-general type. Therefore,  the integral points of $\mathrm{Hilb}_{d,n}$ should not be dense (even though they can be infinite).  

 The expectation that the integral points on $\mathrm{Hilb}_{d,n}$ should not be dense was investigated by  Lawrence-Venkatesh \cite{LV} for large enough $d$ and $n$. The current state-of-the-art   can be stated as follows (see \cite[Proposition~10.2]{LV}). (Note that the following statement provides a non-density statement \emph{only} for $\ZZ[1/S]$-points.)

\begin{theorem}[Lawrence-Venkatesh]\label{thm:lawrence}    There is an integer $n_0$ and a function $D_0(n)$   such that, for every $n\geq n_0$,    every $d\geq D_0(n)$, and every positive integer $S\geq 1$, the   set $\mathrm{Hilb}_{d,n}(\ZZ[1/S])$ is not dense in $\mathrm{Hilb}_{d,n, \Qbar}$.
\end{theorem}

Motivated by Lawrence-Venkatesh's recent breakthrough,  we  show that the non-density of integral points on the Hilbert scheme $\mathrm{Hilb}_{d,n}$ valued in a number field persists to non-density over finitely generated fields.

\begin{theorem}[Main Result, V] \label{thm:non_density_hyps} 
Let $d\geq 3$ be an integer and let $n\geq 2$ be an integer.
Suppose that, for every number field $K$ and every finite set of finite places $S$ of $K$, the set    $\mathrm{Hilb}_{d,n}(\OO_{K,S})$ is not dense in $\mathrm{Hilb}_{d,n}$. Then,   for every $\mathbb{Z}$-finitely generated regular integral domain of characteristic zero $A$, we have that $\mathrm{Hilb}_{d,n}(A)$ is not dense in $\mathrm{Hilb}_{d,n}$.
\end{theorem}

 The proof of Theorem \ref{thm:non_density_hyps} is more involved than the proof of Theorem \ref{thm:hypsurf_intro}. For example, due to the fact that the Hilbert scheme does not admit any period map with finite fibres, we are forced to argue on  the stack $[\mathrm{PGL}_{n+2}\backslash \mathrm{Hilb}_{d,n}]$ and to relate the non-density of the integral points on the stack to that on the Hilbert scheme using finiteness results for $\mathrm{PGL}_{n+2}$-torsors over number rings.  
 
\begin{ack} We thank Kenneth Ascher, Raymond van Bommel, Daniel Loughran, Siddharth Mathur, and Kang Zuo for helpful discussions.
The first named author gratefully acknowledges support from SFB/Transregio 45. The second named author gratefully acknowledges support from the Institute for Advanced Study.
\end{ack}

\begin{con} We let $k$ be an algebraically closed field of characteristic zero. 
A variety over $k$ is a finite type separated scheme over $k$.
 If $X$ is a variety over $k$ and $A\subset k$ is a subring, then a model for $X$ over $A$ is a pair $(\mathcal{X}, \phi)$ with $\mathcal{X}$ a finite type separated  scheme over $A$ and $\phi:\mathcal{X}_k\to X$ an isomorphism of schemes over $k$. We will usually omit $\phi$ from our notation and simply refer to $\mathcal{X}$ as a model for $X$ over $A$.
 
  If $K$ is a number field and $S$ is a finite set of finite places of $K$, we let $\OO_{K,S}$ be the ring of $S$-integers of $K$. 
  
  If $k\subset L$ is a field extension and $X$ is a variety over $k$, we denote  $X\times_{ k} \Spec L$ by $X_L$.
\end{con}

\section{Arithmetic hyperbolicity and geometric hyperbolicity}

  To prove Theorems \ref{thm:hypsurf_intro} and \ref{thm:intro} on the arithmetic hyperbolicity of certain varieties (as defined in Definition \ref{defn:arhyp}), we will use a geometric criterion for the persistence of  arithmetic hyperbolicity of a variety  along field extensions  proven in \cite{JAut}. To state this criterion, we introduce the notion of geometric hyperbolicity. We view this property as a ``function field'' analogue of arithmetic hyperbolicity.

 \begin{definition}\label{def:geom_hyp}
 A variety $X$  over $k$ is \emph{geometrically hyperbolic over $k$} if, for every smooth integral curve $C$ over $k$, every $c$ in $C(k)$, and every $x$ in $X(k)$, the set $\Hom_k((C,c),(X,x))$ of morphisms of $k$-schemes $f:C\to X$ with $f(c)=x$ is finite.
 \end{definition} 
 
 More generally, a finite type separated Deligne-Mumford algebraic stack $X$ over $k$ is \emph{geometrically hyperbolic} if, for every smooth integral curve $C$ over $k$, every $c$ in $C(k)$, and every $x$ in (the groupoid) $X(k)$, the set $\Hom_k((C,c),(X,x))$ of isomorphism classes of morphisms $f:C\to X$ with $f(c)= x$ is finite. (Here $f(c) = x$ means that $f(c)$ and $x$ are isomorphic in $X(k)$.)
  
  \begin{example}[Urata's theorem]\label{example:urata}
A proper variety $X$ over $\CC$ which is Brody hyperbolic (i.e., has no entire curves) is geometrically hyperbolic. Indeed, as $X^{\an}$ is a compact complex-analytic space with no entire curves, it follows from Brody's theorem that $X^{\an}$ is Kobayashi hyperbolic (as defined in \cite{KobayashiBook}). Therefore, as $X^{\an}$ is a compact   Kobayashi hyperbolic complex-analytic space, we conclude that $X$ is geometrically hyperbolic from Urata's theorem   \cite[Theorem~5.3.10]{KobayashiBook} (or the original \cite{Urata}). (Note that Urata's theorem has been extended to the logarithmic case in \cite{JLev}.)
  \end{example}
  
  \begin{example}  \label{exa:cp}
  Let $\mathcal{M}$ be the locally finite type separated Deligne-Mumford algebraic stack of smooth proper canonically polarized varieties over $\mathbb{Q}$. That is, for a scheme $S$ over $\mathbb{Q}$, the objects of $\mathcal{M}(S)$ are smooth proper morphisms $\mathcal{X}\to S$ whose geometric fibres are connected and have ample canonical bundle. (For example, for every $g\geq 2$, the stack of smooth proper genus $g$ curves $\mathcal{M}_g$ is an open and closed substack of $\mathcal{M}$. In fact, $\mathcal{M}$ is the disjoint union of the stacks $\mathcal{M}_h$, where $h$ runs over all polynomials in $\mathbb{Q}[t]$ and  $\mathcal{M}_h$ is the substack of smooth proper canonically polarized varieties with Hilbert polynomial $h$.)  Let $X$ be a quasi-projective scheme over $\mathbb{C}$ such that there exists a quasi-finite morphism $X\to \mathcal{M}_{\CC}$. (In other words, there is a smooth proper morphism $f:Y\to X$ whose geometric fibres are canonically polarized varieties such that, for every $x$ in $X(\CC)$, the set of $y$ in $Y(\CC)$ with $Y_x\cong Y_y$ is finite. In particular, the family $f:Y\to X$ of canonically polarized varieties has ``maximal variation in moduli''.) Then, it follows from Viehweg-Zuo's theorem (see \cite{VZ})   that $X$ is Brody hyperbolic. In particular, if $X$ is projective, then  Urata's theorem (Example \ref{example:urata}) implies that $X$ is geometrically hyperbolic over $\CC$. (It seems reasonable to suspect that the assumption that $X$ is projective is unnecessary; see \cite{JSZ} for recent progress.)
  \end{example}
  
 We will prove the geometric hyperbolicity of some (not necessarily proper) varieties over $\CC$  by appealing to their complex-analytic properties. The following lemma will then be applied to deduce the geometric hyperbolicity of these varieties over every algebraically closed field of characteristic zero.
  
  \begin{lemma}\label{lem:geomhyp_persists} Let $k\subset L$ be an extension of algebraically closed fields of characteristic zero.
  If $k$ is uncountable and $X$ is a finite type separated geometrically hyperbolic Deligne-Mumford algebraic stack over $k$, then $X_L$ is geometrically hyperbolic over $L$.
  \end{lemma}
  \begin{proof} Assume that $X_L$ is not geometrically hyperbolic over $L$. We show that $X$ is not geometrically hyperbolic over $k$. To do so, 
let $C$ be smooth affine connected curve over $L$, let $c\in C(L)$,  and let $x\in X(L)$ be such that the set $\Hom_L((C,c),(X_L,x))$ of (isomorphism classes of) morphisms $f:C\to X_L$ with $f(c) = x$  is infinite. Let $f_1, f_2, \ldots$ be a sequence of pairwise distinct (non-isomorphic) elements in $\Hom_L((C,c),(X_L,x))$. Let $S$ be an integral variety over $k$ and let $(\mathcal{C}, P)$ be a model  for $(C,c)$ over $S$. That is, the morphism $\mathcal{C}\to S$ is a smooth affine geometrically connected morphism of relative dimension one and $P\in \mathcal{C}(S)$ is a section such that there is a  (fixed) isomorphism  $\mathcal{C}_L \cong C$ and $P_L = c$. We now recursively descend every $f_i:C\to X$ to some ``\'etale neighbourhood'' of $S$ (using for instance \cite[Appendix~B]{Rydh2}). Thus, let $A_1\subset L$  be a finitely generated $k$-algebra with $S_1 =\Spec A_1$, let $S_1\to S$  be an \'etale morphism and let $F_1:\mathcal{C}_{S_1}\to X\times_k S_1$ be a morphism  with $F_1(P) =\{x\}\times S_1$ such  that the morphism $f_1:C\to X_L$ coincides with $F_{1,L}:C\cong \mathcal{C}_L\to X_L$.   Now, we construct integral affine varieties $S_2, S_3, \ldots$ over $k$ recursively, as follows. Assume $S_{i-1}$ has been constructed. Then, for every $i=2, 3, \ldots$, we choose a finitely generated $k$-algebra $A_i\subset L$ with $S_i =\Spec A_i$, an \'etale morphism $S_i\to S_{i-1}$  and a morphism $F_i\in \Hom_{S_i}((\mathcal{C}_i,P_i),(X\times_k S_i, x\times \{S_i\}))$ with $\mathcal{C}_i = \mathcal{C}_{S_i}$,  $P_i = P_{S_i}$, and  $F_{i,L} = f_i$   such that, for every  $1\leq j < i$, every $s$ in $S_{j}(k)$ and every $s'$ in $S_{i}(k)$ lying over $s$, the morphism $F_{i, s'}$ does not equal   $F_{j,s}$.  Let $Z_i$ be the (non-empty and open) image of $S_i\to S$. Since $k$ is uncountable and every $Z_i$ is a non-empty open of $S$, there is an $s$ in $S(k)$ contained in $\cap_{i=1}^\infty Z_i$. Now,  for every $i=1,2,\ldots$, let $s_i$ be a point of $S_i(k)$ lying over $s$ in $S(k)$. Define $D:=\mathcal{C}_{s}$ and note that $D\cong \mathcal{C}_{S_i,s_i}$. Moreover,  the morphisms  $F_{i,s_i}:D\cong \mathcal{C}_{S_i,s_i}  \to X\times \{s_1\}\cong X$  are, by construction,   pairwise distinct. Finally, as $F_{i,s_i}(P_{s_i}) = x$, we see that $X$ is not geometrically hyperbolic over $k$, as required. 
  \end{proof}
 
  \begin{remark}[From pointed curves to pointed varieties]\label{remark:geomhyp}
The notion of geometric hyperbolicity is studied in more generality in \cite{JXie} building on \cite{JKa} (see also \cite{vBJK, JBook}). For example, let $k$ be an \emph{uncountable} algebraically closed field of characteristic zero, and let $X$ be a finite type separated Deligne-Mumford algebraic stack over $k$. Then, for every reduced variety $Y$ over $k$, every $y$ in $Y(k)$, and every $x$ in $X(k)$, the set $\Hom_k((Y,y),(X,x))$ of morphisms $f:Y\to X$ with $f(y) = x$ is finite. Indeed, suppose that $f_1,f_2,\ldots$ are pairwise distinct morphisms from $Y$ to $X$ which map $y$ to $x$. Let $Y^{i,j}\subset Y$ be the closed subset    of points  $P$   such that $f_i(P) = f_j(P)$.   Let $w$ be a point of $Y(k)$ such that, for every $i\neq j$, the point $w$  does not lie in $Y^{i,j}$.  (Such a point exists as $k$ is uncountable and $Y^{i,j}\neq Y$ whenever $i\neq j$.) Let $C\subset Y$ be a smooth curve containing $w$ and $y$. Then the morphisms $f_1|_C, f_2|_C, \ldots$ are   pairwise distinct morphisms from $C$ to $X$ and send $y$ to $x$. This shows that $X$ is not geometrically hyperbolic, as required.
  \end{remark}

  \begin{remark}[Descending along coverings]
  \label{remark:cw_geomhyp} Let $X\to Y $ be a finite \'etale morphism of finite type separated Deligne-Mumford algebraic stacks over $k$. Then $X$ is geometrically hyperbolic over $k$ if and only if $Y$ is geometrically hyperbolic over $k$. This is proven when $X$ and $Y$ are projective schemes in \cite[\S 5]{JKa}, and the arguments in \emph{loc. cit.} easily adapt to prove the more general statement for stacks.
  \end{remark}
  
 The relation between geometric hyperbolicity and arithmetic hyperbolicity  is provided by the following consequence of the results proven in \cite[\S 4]{JAut}.

 \begin{proposition}\label{prop:gen0}
  Let $k\subset L$ be an extension of algebraically closed fields of characteristic zero, and let $X$ be an arithmetically hyperbolic variety over $k$ such that $X_L$ is geometrically hyperbolic over $L$. Then,  $X_L$ is arithmetically hyperbolic over $L$.
 \end{proposition}

\section{Weakly bounded varieties and persistence of non-density}\label{section:nd} 
To prove Theorem \ref{thm:hypsurf_intro} on the arithmetic hyperbolicity of the moduli of smooth hypersurfaces, we will use the geometric hyperbolicity of the moduli stack of smooth hypersurfaces. However, 
to prove Theorem \ref{thm:non_density_hyps} (which is concerned with the non-density of integral points on  a certain Hilbert scheme), we will require an additional property of the moduli space. Namely, we will need that it is ``weakly bounded''.
Here we follow the terminology of  Kov\'acs-Lieblich; see \cite{KovacsLieblich}. To be precise, we use the notion  of ``weak boundedness'' to give a criterion for extending results on non-density of integral points valued in number fields to non-density of integral points valued in finitely generated fields (see Theorem \ref{thm:criterion_nd}).
  
 \begin{definition}[Kov\'acs-Lieblich]\label{def:wb1} Let $\overline{X}$ be a   projective scheme over $k$, let $\mathcal{L}$ be an ample line bundle on $\overline{X}$, and let $X\subset \overline{X}$ be a dense open subscheme. 
We say that $X$ is \emph{weakly bounded over $k$ in $\overline{X}$ with respect to $\mathcal{L}$} if, for every integer $g\geq 0$, and every $d\geq 0$, there  is a real number $\alpha(X,\overline{X}, \mathcal{L},g,d)$ such that, for every smooth projective connected curve $\overline{C}$ over $k$ of genus $g$ and   every dense open subscheme $C\subset \overline{C}$ with $\# (\overline{C}\setminus C) = d$ and   every morphism $f:C\to X$, the following inequality
\[ 
\deg_{\overline{C}} \overline{f}^\ast \mathcal{L} \leq \alpha(\mathcal{L}, g,d)
\] holds, where $\overline{f}:\overline{C}\to \overline{X}$ is the unique morphism restricting to $f:C\to X$.
 \end{definition}
 
For $Y$ and $X$ projective schemes over $k$, we let $\underline{\Hom}_k(Y,X)$ be the moduli scheme parametrizing morphisms $Y\to X$; recall that $\underline{\Hom}_k(Y,X)$ is a disjoint union of quasi-projective schemes over $k$ (see \cite[\S2]{DebarreBook1}).
We will make use of the following basic proposition.
 \begin{proposition}\label{thm:kl}  
 Let $\overline{X}$ be a projective variety over $k$, let $\mathcal{L}$ be an ample line bundle on $\overline{X}$, and let $X\subset \overline{X}$ be a dense open subscheme. Let $\overline{C}$ be a smooth projective curve and let $C\subset\overline{C}$ be a dense open subscheme. If $X$ is weakly bounded over $k$ in $\overline{X}$ with respect to $\mathcal{L}$, then  
 $\Hom_k(C,X)$ is a   quasi-compact constructible subset of $\underline{\Hom}_k(\overline{C},\overline{X})(k)$.
 \end{proposition}
 \begin{proof}
 
Let $g$ be the genus of $\overline{C}$ and let $d:=\#(\overline{C}\setminus C)$. Let $\alpha:=\alpha(\mathcal{L},g,d)$ be the real number in Definition \ref{def:wb1}, and note that $\Hom_k(C,X)$ is a subset of  the scheme $\underline{\Hom}_k^{\leq \alpha}(\overline{C},\overline{X})$ parametrizing morphisms $\overline{C}\to \overline{X}$ of degree at most $\alpha$ (with respect to $\mathcal{L}$).

 Consider  the natural morphisms of finite type schemes
 \[
\mathrm{ev}\colon C\times  \Hom^{\leq \alpha}_k(\overline{C},\overline{X}) \to \overline{X}, \quad \mathrm{pr}:C\times  \Hom_k^{\leq \alpha}(\overline{C},\overline{X}) \to  \Hom_k^{\leq \alpha}(\overline{C},\overline{X}).
 \] Let $\Delta$ be the boundary of $X$ in $\overline{X}$. Let $\Delta' := \mathrm{ev}^{-1}\Delta$ be the (closed) inverse image of $\Delta$ in the finite type $k$-scheme $C\times \Hom_k^{\leq\alpha}(\overline{C},\overline{X})$. Let  $Z:=\mathrm{pr}(\Delta') $ be the   image of $\Delta'$ in $\Hom_k^{\leq \alpha}(\overline{C},\overline{X})$, and note that $Z$ is constructible \cite[Tag~054J]{stacks-project}. As the complement of $Z$ is a constructible subset of the finite type $k$-scheme $\underline{\Hom}_k^{\leq\alpha}(\overline{C},\overline{X})$ whose $k$-points equal   $\Hom_k(C,X)$, this shows that $\Hom_k(C,X)$ is a finite union of locally closed subschemes of $\underline{\Hom}_k^{\leq \alpha}(\overline{C},\overline{X})$.
 \end{proof}

% \begin{remark} In the statement of Proposition \ref{thm:kl}, 
% it seems reasonable to suspect that $\Hom_k(C,X)$ is in fact locally closed in $\underline{\Hom}_k(\overline{C},\overline{X})(k)$ (and not merely constructible). As we do not need this additional structure of $\Hom_k(C,X)$ (in this paper), we will not discuss it any further. (However, provided $k=\CC$ and  $X$ is ``hyperbolically embedded in $\overline{X}$'', then it is shown in \cite{JLev} that $\Hom_{\CC}(C,X)$ is in fact a closed   complete subvariety of $\Hom_{\CC}(\overline{C},\overline{X})(\CC)$.)
% \end{remark}
% 

 \begin{definition}\label{def:wb}
A quasi-projective scheme $X$ over $k$ is \emph{weakly bounded over $k$} if there exists a projective scheme  $\overline{X}$  over $k$, an ample line bundle $\mathcal{L}$  on $\overline{X}$, and an open immersion $X\subset \overline{X}$ such that $X$ is weakly bounded over $k$ in $\overline{X}$ with respect to $\mathcal{L}$.
 \end{definition}
 
 \begin{proposition}\label{thm:kl2}
 Let $X$ be a weakly bounded quasi-projective scheme over $k$, and let $C\subset \overline{C}$ be a dense open of a smooth projective connected curve $\overline{C}$ over $k$. Then,  for every projective variety $\overline{X}$ over $k$ and every open immersion $X\to \overline{X}$, the subset $\Hom_k(C,X)$  of $\underline{\Hom}_k(\overline{C},\overline{X})(k)$ is quasi-compact and constructible.
 \end{proposition}
 \begin{proof}
 Let $\overline{X}'$ be such that $X$ is weakly bounded in $\overline{X}'$ with respect to some ample line bundle. Then, $\Hom(C,X)$ is quasi-compact constructible in $\underline{\Hom}_k(\overline{C},\overline{X}')(k)$ by Proposition \ref{thm:kl}. Now, to prove the proposition, we choose a projective variety $Z$, an open immersion $X\subset Z$,  a proper birational morphism $Z\to \overline{X}'$
  which is an isomorphism over $X$, and a proper birational morphism $Z\to \overline{X}$ which is an isomorphism over $X$.
  Note that, the inverse image  of the quasi-compact constructible subset  $\Hom_k(C,X)\subset\Hom_k(\overline{C},\overline{X}')(k)$ in $\underline{\Hom}_k(\overline{C},Z)(k)$
along the finitely presented morphism of $k$-schemes $\underline{\Hom}_k(\overline{C},Z)\to \underline{\Hom}_k(\overline{C},\overline{X}')$ is again a quasi-compact   constructible subset and equals $\Hom_k(C,X)$.  Moreover,  the image of the latter quasi-compact constructible subset in $\underline{\Hom}_k(\overline{C},\overline{X})(k)$ along the finitely presented morphism of  $k$-schemes is again equal to $\Hom_k(C,X)$ and quasi-compact constructible (by Chevalley's theorem). This concludes the proof.
    \end{proof}

 \begin{lemma}\label{lem:nond}
  Let $\ZZ\subset A$ be a finitely generated integral domain of characteristic zero and let $\mathcal{X}$ be a finite type scheme over $A$. Let $k:=\overline{\mathrm{Frac}(A)}$ be an algebraic closure of $\mathrm{Frac}(A)$. Let $k\subset L$ be an extension of algebraically closed fields with $L$ of transcendence degree one over $k$. Assume that $X:=\mathcal{X}_k$ is  quasi-projective   over $k$.  Assume the following two properties hold.
 \begin{enumerate}
 \item The  variety $X_L$ is weakly bounded and
 geometrically hyperbolic over $L$.
 \item For every finitely generated subalgebra $A'\subset k$ containing $A$, the set  $$\mathcal{X}(A') = \Hom_A(\Spec A', \mathcal{X})$$ is not dense in $X$. 
 \end{enumerate}
 Then, for any finitely generated subring $B\subset L$ containing $A$, the set $\mathcal{X}(B) $ is not dense in $X(L)$.
 \end{lemma}
 
 \begin{proof}  
To prove the statement, let $B\subset L$ be a finitely generated subring containing $A$ and define $K :=\mathrm{Frac}(B)$.  We now show that $\mathcal{X}(B)$ is not dense in $X(L)$. 

Note that if $K$ has transcendence degree zero over $\mathrm{Frac}(A)$, then it follows from $(2)$ that $\mathcal{X}(B)$ is not dense in $X(L)$. Therefore, to prove the lemma, we may and do assume that $K=\mathrm{Frac}(B)$ has transcendence degree one over $\mathrm{Frac}(A)$. Moreover, replacing  $A$ by a finitely generated sub-$A$-algebra of $k$, we may and do assume that the scheme $\mathcal{C}:=\Spec B$ over $\Spec A$ has  a section $\sigma:\Spec A \to \mathcal{C}$ and that $\mathcal{C}\to \Spec A$ is a smooth morphism.

 Define $C:=\mathcal{C}_k = \mathcal{C}\times_A k$, and note that $C$ is a smooth affine   curve over $k$. Furthermore, since $\mathcal{C}(A)\neq \emptyset$ and $\mathcal{C}$ is an integral scheme, it follows that $C$ is connected.
    We let $\sigma_k$ be the $k$-rational point of $C$ induced by the section $\sigma:\Spec A \to \mathcal{C}$; we will also view this as an $L$-point of $C_L$. Let $\overline{C}$ be the smooth projective connected model of $C$ over $k$. Now,  let $\Delta$ be the closure of the subset $\mathrm{Im}[ \mathcal{X}(A) \to  X(k)]$ in $X$ with the reduced closed subscheme structure,  and note that  $\Delta\subsetneq  X$ is a proper closed subscheme  by our second assumption $(2)$.
    
  Let $\overline{X}$  be a projective variety over $k$ with $X\subset \overline{X}$ an open immersion.   Since $X_L$ is weakly bounded over $L$, it follows that  $\Hom_L(C_L,X_L)$ is a quasi-compact constructible subset  of $\underline{\Hom}_L(\overline{C}_L,\overline{X}_L)(L)$; see Proposition \ref{thm:kl2}.
We define $Z\subset \Hom_L(C_L,X_L)(L)$ to be the closure of \[\mathrm{Im}[\mathcal{X}(\mathcal{C})\to \Hom_L(C_L,X_L)(L)]\] in $\Hom_L(C_L,X_L)(L)$.   Since $Z\subset \Hom_L(C_L,X_L)(L)$ is closed, it follows that $Z$ is a quasi-compact constructible subset of $\underline{\Hom}_L(
\overline{C}_L,\overline{X}_L)(L)$.

  Note that the evaluation map $Z\to X(L)$ which sends $f$ to $f(\sigma_k)$ has finite fibres by the geometric hyperbolicity of $X_L$ over $L$. Moreover, since the dense subset $\mathrm{Im}[\mathcal{X}(\mathcal{C})\to \Hom_L(C_L,X_L)]$ of $Z$ lands in $\Delta$, we see that the evaluation map $Z\to X(L)$ factors through $\Delta(L)$.
  In particular, it follows that $$\dim Z \leq \dim \Delta.$$  Therefore, since $\dim \Delta <\dim X   $, we see that
  \begin{eqnarray*}\label{ineqqq} \dim Z &<& \dim X. \end{eqnarray*} 
 Consider the Cartesian diagram of  morphisms of schemes 
  \[
  \xymatrix{
  \underline{\Hom}_L(\overline{C}_L,\overline{X}_L) \ar[rr] \ar[d]  & & \overline{X}\times_k \Spec L \ar[d] 
  \\ \underline{\Hom}_k(\overline{C},\overline{X}) \times_k  \overline{C} \ar[rr] & & \overline{X} \times \overline{C},
  }
  \] where the bottom horizontal arrow is given by the universal evaluation morphism $(f,c)\mapsto (f(c),c)$, and the right vertical morphism is induced by the geometric generic point $\Spec L\to \overline{C}$ of $\overline{C}$. 
  Since $Z$ is a quasi-compact constructible subset of  the $L$-points $\underline{\Hom}_L(\overline{C}_L,\overline{X}_L)(L)$ of the scheme $\underline{\Hom}_L(\overline{C}_L,\overline{X}_L)$, its image $Z'$ along the morphism $  \underline{\Hom}_L(\overline{C}_L,\overline{X}_L) \to \overline{X}_L $ is a quasi-compact constructible subset of $\overline{X}(L)$.
 Note that $Z'\subset X(L)$ and that $\dim Z'\leq \dim Z <\dim X$. Since $\mathrm{Im}[\mathcal{X}(\mathcal{C})\to X(L)]$ is contained in $Z'$ and $\dim Z'<\dim X$, we conclude that  $\mathcal{X}(\mathcal{C})$ is not dense in $X(L)$, as the dimension of the closure of $Z'$ is at most that of $Z'$, by constructibility.
 \end{proof}

 The following result provides a general criterion for proving the persistence of non-density of integral points on an algebraic variety. In fact, in  Theorem \ref{thm:geom_hyp_2}, we verify that a variety with a quasi-finite period map verifies the first property necessary to apply this result.
 
 \begin{theorem}\label{thm:criterion_nd}  Let $ A$ be a finitely generated integral domain of characteric zero and let $\mathcal{X}$ be a finite type scheme over $A$. Let $k:=\overline{\mathrm{Frac}(A)}$ be an algebraic closure of $\mathrm{Frac}(A)$, and let $k\subset L$ be an extension of algebraically closed fields.  Let $X:=\mathcal{X}\times_A k$.  Assume the following two properties hold.
 \begin{enumerate}
 \item The variety $X_L$ is weakly bounded and
 geometrically hyperbolic over $L$.
 \item For every finitely generated subring $A'\subset k$ containing $A$, the set  $$\mathcal{X}(A') = \Hom_A(\Spec A', \mathcal{X})$$ is not dense in $X$. 
 \end{enumerate}
 Then, for any finitely generated subring $B\subset L$ containing $A$, the set $\mathcal{X}(B) $ is not dense in $X_L$.
 \end{theorem}
 \begin{proof} Let $K$ be the algebraic closure of $\mathrm{Frac}(B)$ in $L$, and note that $K$ has finite transcendence degree over $k$. We proceed by induction on the transcendence degree $d$ of $K$ over $k$. If  $d=0$, then the required non-density statement holds by  $(2)$. Now, assume $d>0$ and  let $K_0\subset K$ be an algebraically closed subfield of transcendence degree $d-1$ over $k$. Define $Y:= X_{K_0}$.   Now, as $X_L$ is weakly bounded and geometrically hyperbolic over  $L$, we have that $Y_K$ is weakly bounded and geometrically hyperbolic over $K$.  Moreover, write $\mathcal{Y} =\mathcal{X}$ (for the sake of clarity) and note that,  by the induction hypothesis,  for every finitely generated subring $A'\subset K_0$ containing $A$, the set $\mathcal{Y}(A')$ is not dense in $Y$.
 Therefore, as $K$ has transcendence degree one over $K_0$, we conclude that $X_K=Y_K$ satisfies the required non-density statement (Lemma \ref{lem:nond}). This concludes the proof.
 \end{proof}

\section{Hodge theory}\label{section:hodge_theory}
We set notation by recalling the definition of a period domain, following \cite[Section 3]{Schmid}. Let $H$ be a finitely-generated free $\mathbb{Z}$-module, $k$ an integer, and $\{h^{p,k-p}\}$ a collection of non-negative integers with $h^{p,k-p}=h^{k-p, p}$ for all $p$, such that $$\sum_p h^{p, k-p}=\text{rk}_{\mathbb{Z}} H.$$ Let $\hat{\mathscr{F}}$ be the flag variety parametrizing decreasing, exhaustive, separated filtrations of $H_\mathbb{C}$, $(F^\bullet)$, with $\dim F^p= \sum_{i\geq p} h^{i, k-i}$.

Let $\mathscr{F}\subset \hat{\mathscr{F}}$ be the analytic open subset of $\hat{\mathscr{F}}$ parametrizing those filtrations corresponding to $\mathbb{Z}$-Hodge structures of weight $k$, i.e.~those filtrations with $$H_\mathbb{C}=F^p+\overline{F^{k-p+1}}$$ for all $p$.

Now suppose $q$ is a non-degenerate bilinear form on $H_{\mathbb{Q}}$, symmetric if $k$ is even and skew-symmetric if $k$ is odd.  Let $D\subset \mathscr{F}$ be the locally closed analytic subset of $\mathscr{F}$ consisting of filtrations corresponding to polarized Hodge structures (relative to the polarization $q$), i.e.~the set of filtrations $(F^\bullet)$ in $\mathscr{F}$ with $$q_{\mathbb{C}}(F^p, F^{k-p+1})=0 \text{ for all $p$}$$ and $$q_{\mathbb{C}}(Cv, \bar v)>0$$ for all nonzero $v\in H_{\mathbb{C}}$, where $C$ is the linear operator defined by $C(v)=i^{p-q}v$ for $$v\in H^{p,q}:= F^p\cap \overline{F^q}.$$

Let $G=O(q)$ be the orthogonal group of $q$; it is a $\mathbb{Q}$-algebraic group. We abuse notation to denote $G_{\mathbb{Z}}=G(\mathbb{Q})\cap GL(H)$.  Let $\Gamma\subset G_{\mathbb{Z}}$ be a finite index subgroup. Let $h$ be a point of  the complex-analytic space $\Gamma\backslash D$.

 Let $X$ be an integral   variety over $k$. A complex-analytic map to $f: X\to \Gamma\backslash D$ is \emph{locally liftable} if it locally factors through the quotient map $D\to \Gamma\backslash D$. If $X$ is smooth and $f$ is locally liftable, we say $f$ is \emph{horizontal} if for each $x\in X$ and each tangent vector $v$ in $T_xX$, we have that for a local lifting $\tilde f$ of $f$, $\tilde f_*(v)\in T_{\tilde f(x)}D$ sends $F_{\tilde f(x)}^p$ into $F_{\tilde f(x)}^{p-1}$, for each $p$. Here we regard the tangent space to $D$ at a point $d$ as an element of $$\text{End}(H_{\mathbb{C}})/\text{Lie}(\text{Stab}_{G_{\mathbb{C}}}(F_d^\bullet)).$$ Horizontality is equivalent to the statement that the associated variation of Hodge structure satisfies Griffiths transversality. See \cite[Section 3]{Schmid} for details.

 If $X$ is smooth, then we say that a holomorphic map $X^{\an}\to \Gamma\backslash D$ is \emph{a period map} if it is locally liftable and horizontal (i.e., satisfies Griffiths transversality). More generally, a holomorphic map $X^{\an}\to \Gamma\backslash D$ is a period map if there is a desingularization $\tilde{X}\to X$ such that the composed morphism $\tilde{X}^{\an}\to X^{\an}\to \Gamma\backslash D$ is a period map.

\begin{proposition}[Rigidity Theorem] \label{rigidity-theorem}  Let $D$, $\Gamma$, $h$, and $X$ be as above.
Let $f:X^{\an}\to \Gamma\backslash D$ be a period map and let $g:X^{\an}\to \Gamma\backslash D$ be a period map. Let $x\in X^{\an}$ such that $f(x) = g(x)=h$. Assume that the monodromy representation $f_\ast:\pi_1(X,x)\to \pi_1(\Gamma\backslash D, h)$ of $f$ is isomorphic to the   monodromy representation $g_\ast: \pi_1(X,x)\to \pi_1(\Gamma\backslash D, h)$ of $g$. Then, we have that $f=g$.
\end{proposition}
\begin{proof}
This is the so-called rigidity theorem of Deligne-Griffiths-Schmid \cite[7.24]{Schmid}. (This was proven in \cite[Section~4]{DeligneII}  in the case that the variation of Hodge structures comes from geometry.)
\end{proof}

The following result is based on ``Arakelov's inequality''. Historically, this started with Arakelov-Parshin \cite{ArakelovShaf} for families of curves of genus at least two, and was then subsequently generalized to more general variations of Hodge structures  by  Deligne, Faltings, Peters, and Jost-Zuo; see \cite{ FaltingsArakelov, Peterss, JostZuo}. However, to obtain the desired statement for varieties with a quasi-finite period map, we need \emph{in addition} to the aforementioned results the recently established algebraization theorem of Bakker-Brunebarbe-Tsimerman \cite{BakkerBrunebarbeTsimerman}.

\begin{theorem}[Consequence of Arakelov's inequality]\label{thm:arakelov}
Let $D$, $\Gamma$, and $X$ be as above. Assume $X$ is quasi-projective over $\CC$ and let $p:X^{\an}\to \Gamma\backslash D$ be a  period map with finite fibres.  Then $X$ is weakly bounded over $\CC$ (see Definition \ref{def:wb}).
\end{theorem}
\begin{proof} If $X$ is smooth, then Theorem \ref{thm:arakelov} can be deduced directly from recent results of Deng \cite{Deng}. Indeed, in \emph{loc. cit.} Deng shows that $X$ is in fact "algebraically hyperbolic". If $X$ is singular, we argue as follows.

Let $g: C\to X$ be a map as in Definition \ref{def:wb}. The map $f\circ g$ classifies a variation of Hodge structure on $C$; let $\mathscr{V}^{p, k-p}$ be the associated Hodge bundle, and let $\mathscr{F}^p$ be the associated filtration. Consider the Griffiths bundle $$\mathscr{L}=\bigotimes_i \det(\mathscr{F}^i).$$ By   \cite[Theorem~6.2]{BakkerBrunebarbeTsimerman}, the line bundle $\mathscr{L}$ is ample. Thus, to conclude the proof, it suffices to show that $X$ is weakly bounded (in some compactification) with respect to $\mathscr{L}$.  

To show that the desired  bound exists, we first appeal to a result of Peters. In fact, by \cite[Theorem 3.1]{Petersa}, we have that,  for each $p$, the integer $\deg \mathscr{V}^{p,k-p}$ is bounded in terms of only the genus of $\overline{C}$ and $\#(\overline{C}\setminus C)$ (as well as invariants of the variation of Hodge structure on $X$, which is fixed).

Now, it follows from the aforementioned result of Peters that the degree of $C$ with respect to the auxiliary variation of Hodge structures $\otimes_{p\in \mathbb{Z}} \Lambda^{r_p} \mathscr{V}$ with $r_p = \mathrm{rank} \ \mathscr{F}^p$  is bounded by a constant depending only on $\overline{C}$, $\#(\overline{C}\setminus C)$, $n$, and invariants of $\mathscr{V}$. In particular, since the Griffiths line bundle $\mathcal{L}$ is the lowest piece of the Hodge filtration of   $\otimes_{p\in \mathbb{Z}} \Lambda^{r_p} \mathscr{V}$, 
% Now, by  \cite[Theorem~6.2]{BakkerBrunebarbeTsimerman},  there is an integer $n\geq 1$ such that the sections of $\mathscr{L}^{\otimes n}$ vanishing at the boundary (see \cite[Definition~6.1]{BakkerBrunebarbeTsimerman}) realize $X$ as a quasi-projective scheme. Unraveling the definitions, this implies that $\mathscr{L}^{\otimes n}$ is a subbundle of $(\oplus_{p,k}\mathcal{V}^{p,k-p})^{\otimes n}$.  Since the degree of $C$ with respect to $(\oplus_{p,k}\mathcal{V}^{p,k-p})^{\otimes n}$  is bounded by a constant depending only on $\overline{C}$, $\#(\overline{C}\setminus C)$, $n$, and invariants of the variation of Hodge structures on $X$,
 the same holds for the degree of $C\to X$ with respect to $\mathcal{L}$. We conclude that $X$ is weakly bounded, as required. 
 \end{proof}

 \begin{remark} 
In the proof of Theorem \ref{thm:arakelov} we appeal to the recent work of Bakker-Brunebarbe-Tsimerman \cite{BakkerBrunebarbeTsimerman} to guarantee the ampleness of $\mathscr{L}$.  If $X$ is smooth, we expect that this fact can also be deduced from the methods of Sommese's classical paper \cite{Sommese78}.   
\end{remark}

\section{Proof of Theorems   \ref{thm:intro} and \ref{thm:deligne_intro} }
The geometric finiteness property we require in this paper to prove the persistence of arithmetic hyperbolicity for varieties with a quasi-finite period map is provided by the following finiteness theorem; see   Definition \ref{def:geom_hyp} for the definition of geometric hyperbolicity.

We stress that the following theorem is proven by combining    Deligne's finiteness result for monodromy representations \cite{Delignemonodromy} with Deligne-Grifitths-Schmid's `Rigidity Theorem'' (Proposition \ref{rigidity-theorem}).  Deligne's finiteness theorem for monodromy representations used below generalizes the result of Faltings for weight one monodromy representations \cite{FaltingsArakelov}, but itself does not immediately imply the desired finiteness result we require.

\begin{theorem} [Deligne + Rigidity Theorem]\label{thm:deligne}  
Let $X$ be a variety  over $k$ which admits (up to Galois conjugation) a quasi-finite period map. Then $X$ is geometrically hyperbolic over $k$.
\end{theorem}
\begin{proof}  We may and do assume that $k=\CC$ (by Lemma \ref{lem:geomhyp_persists}) and that $X$ admits a period map $X^{\an}\to \Gamma\backslash D$ with finite fibres.
We wish to show that for each $c$ in $  C$ and $x$ in $X$, the set   $\text{Hom}((C,c), (X,x))$ of maps $\phi:C\to X$ with $\phi(c) =x$ is finite.  The map $$\text{Hom}((C,c), (X,x))\to \text{Hom}((C,c), (\Gamma \backslash D, f(x)))$$ has finite fibers, because $f$ is quasi-finite. Thus it is enough to show that there are finitely many locally liftable, horizontal analytic maps $$f': (C,c)\to (\Gamma \backslash D, f(x)).$$ We denote the set of such maps by $$\text{Hom}_{per}((C,c), \Gamma\backslash D, f(x)).$$ 
By Proposition \ref{rigidity-theorem} above, the map $$\text{Hom}_{per}((C,c), \Gamma\backslash D, f(x))\to \text{Hom}(\pi_1(C,c), \pi_1(\Gamma \backslash D, f(x)))$$ has finite fibers, so it suffices to show that it has finite image. But this is precisely Deligne's finiteness theorem  \cite{Delignemonodromy}.
\end{proof}

\begin{remark} \label{remark:ds} Let $g\geq 8$ be an integer, and let $X$ be the fine moduli space of principally polarized $g$-dimensional abelian varieties with full level $3$ structure over $\mathbb{C}$. 
In \cite{FaltingsArakelov} Faltings showed that there is a smooth curve $C$  such that the set of non-constant morphisms $f:C\to X$ is infinite. Therefore, as $X$ admits a quasi-finite period map, this shows that one can not expect a strengthening of  Theorem \ref{thm:deligne} for   maps $C\to X$. That is, in Theorem \ref{thm:deligne} one needs to consider maps of \emph{pointed} varieties to obtain finiteness.
\end{remark}

 \begin{remark} Suppose that $X$ is a proper scheme over $\CC$ and that $X$ admits a quasi-finite period map. In this case (as $X$ is proper), there is a different proof of Theorem \ref{thm:deligne}.   Indeed, if $X$ admits a quasi-finite period map, then $X$ has no entire curves \cite[Corollary~9.4]{GriffithsSchmid}. Therefore, as $X$ is also proper, it follows from   Urata's theorem  (Example \ref{example:urata}) that $X$ is geometrically hyperbolic.   
\end{remark}

\begin{theorem}[Deligne + Rigidity Theorem + Arakelov inequality]\label{thm:geom_hyp_2}  If  $X$ is a variety over $k$ which admits  a quasi-finite      complex-analytic period map (Definition \ref{defn:periodmap}), then $X$ is weakly bounded over $k$ and, for every variety $Y$ over $k$, every $y\in Y(k)$, and every $x\in X(k)$, the set of morphisms $f:Y\to X$ with $f(y) = x$ is finite.
\end{theorem}
 
 \begin{proof} We first prove that $X$ is geometrically hyperbolic over $k$. 
 
 If $k=\mathbb{C}$, then Theorem \ref{thm:deligne} says that $X$ is geometrically hyperbolic over $\mathbb{C}$. By a standard specialization argument and Lemma \ref{lem:geomhyp_persists}, we have that $X_L$ is geometrically hyperbolic over any algebraically closed field extension $L$ of $k$ (without any additional assumption on $k$). 
 
 Choosing $L$ to be an uncountable algebraically closed field extension of $k$, it follows from  Remark \ref{remark:geomhyp} that,   for every variety $Y$ over $L$, every $y\in Y(L)$, and every $x\in X(L)$, the set of morphisms $f:Y\to X_L$ with $f(y) = x$ is finite. (As this holds over $L$, it certainly also holds  over $k$, as required.)
 
 To conclude the proof, it suffices to note that    $X$ is weakly bounded over $k$ by   Theorem \ref{thm:arakelov}.
 \end{proof}

\begin{proof}[Proof of Theorem \ref{thm:deligne_intro}]
This is part of the statement of  Theorem \ref{thm:geom_hyp_2}.
\end{proof}

\begin{proof}[Proof of Theorem \ref{thm:intro}]   Combine  Proposition \ref{prop:gen0} and Theorem \ref{thm:geom_hyp_2}.
\end{proof}

\begin{proof}[Proof of Theorem \ref{thm:ls}] Replacing $K$ by a finite field extension if necessary, we may choose   an ample divisor $\mathcal{D}$ in $\mathcal{A}$ such that $\mathcal{D}_K = D$.
Let $\mathcal{M}$ be the (fine) moduli space of smooth hypersurfaces in $\mathcal{A}$ over $\mathcal{O}_{K,S}$ representing $\mathcal{D}$, and note that $\mathcal{M}$ is a quasi-projective scheme over $\mathcal{O}_{K,S}$. Lawrence-Sawin showed that $\mathcal{M}_{\Qbar}$ is arithmetically hyperbolic over $\Qbar$ and that $\mathcal{M}_{\CC}$ admits a quasi-finite period map (see \cite[Theorem~1.1]{LS} and \cite[Proposition~5.10]{LS}, respectively). Therefore, by Theorem \ref{thm:intro}, the variety $\mathcal{M}_k$ is arithmetically hyperbolic over any algebraically closed field $k$ of characteristic zero. This implies that $\mathcal{M}(R) $ is finite, as required.
\end{proof}

%
%\begin{proposition}\label{prop:qf_per} \ari{Maybe integrate this into 5.4?} Let $X$ be a  quasi-projective variety over $k$, and let $k\subset L$ be an extension of algebraically closed fields.
%If $X$ admits (up to Galois conjugation) a quasi-finite complex-analytic period map over $k$, then $X_L$ is geometrically hyperbolic and weakly bounded over $L$.
%\end{proposition}
%\begin{proof} We may  assume that $k=L$.
%Since $X$ admits a quasi-finite period map, it follows from Theorem   \ref{thm:deligne} that $X$ is geometrically hyperbolic over $k$. To show that $X$ is weakly bounded over $k$, we may and do assume that $k=\CC$. The required statement then follows from Arakelov's inequality (Theorem \ref{thm:arakelov}). 
%\end{proof}
%

\section{Locally symmetric varieties and Shimura varieties}
A (smooth connected) variety $X$ over $\CC$ is \emph{locally symmetric} if there exists a bounded symmetric domain $D$, a torsionfree arithmetic subgroup of $\mathrm{Aut}(D)$ and an isomorphism of complex analytic spaces $X^{\an}\cong \Gamma\backslash D$; note that $D$ is biholomorphic to the universal cover of $X^{\an}$.   By Baily-Borel's theorem, a locally symmetric variety is   quasi-projective      over $\CC$. 
  We stress that, with our definition, the affine line $\mathbb{A}^1_{\CC}$ is not a locally symmetric variety as $\mathrm{SL}_2(\ZZ)$ is not torsionfree. In fact, standard results in complex analysis imply that a smooth   quasi-projective connected curve over $\CC$ is a locally symmetric variety (in the above sense) if and only if it is  (Kobayashi) hyperbolic (in the sense of \cite[Chapter~3.2]{KobayashiBook}).

If $X$ is a variety over $k$, then $X$ is a \emph{locally symmetric variety over $k$} if there  exists a subfield $k_0\subset k$, an embedding $k_0\to \CC$, a variety $X_{0}$ over $k_0$, and an isomorphism of $k$-schemes $X_{0,k}\cong X$ such that $X_{0,\CC}$ is a locally symmetric variety over $\CC$ (as defined above).

Let $k\subset L$ be an extension of algebraically closed fields of characteristic zero, and let $X$ be a locally symmetric variety over $k$. Then, as $X$ admits a quasi-finite period map,  by combining  Theorem \ref{thm:intro},  Theorem \ref{thm:criterion_nd}, and Theorem \ref{thm:geom_hyp_2}, we obtain the following result.

\begin{theorem}\label{thm:lsv}  The following statements hold.
\begin{enumerate} 
\item  The locally symmetric variety $X_L$ is geometrically hyperbolic over $L$ and weakly bounded over $L$.
\item If $X$ is arithmetically hyperbolic over $k$, then $X_L$ is arithmetically hyperbolic over $L$.
\item Let $A\subset k$ be a finitely generated $\ZZ$-algebra, and let $\mathcal{X}$ be a model for $X$ over $A$. Assume that, for every finitely generated subring $A'\subset k$ containing $A$, the set $\mathcal{X}(A')$ is not dense in $X$. Then, for any finitely generated subring $B\subset L$ containing $A$, the set $\mathcal{X}(B)$ is not dense in $X_L$.
\end{enumerate}  
\end{theorem}

\begin{remark} Examples of locally symmetric quasi-projective varieties are Shimura varieties (associated to a \emph{torsionfree} congruence subgroup of $\Aut(D)$).  
\end{remark}

 \section{The moduli of smooth hypersurfaces}\label{section:hyps}
 In this section we prove Theorem \ref{thm:hypsurf_intro}. Throughout this section, 
let $d\geq 3$ be an integer, let $n\geq 2$  be  an integer, and let $\mathcal{C}_{d;n} = [\mathrm{PGL}_{n+2}\backslash \mathrm{Hilb}_{d,n}]$ be the stack of smooth  hypersurfaces in  $\mathbb{P}^{n+1}$ of degree $d$. The stack $\mathcal{C}_{d;n}$ is a finite type separated Deligne-Mumford algebraic stack over $\ZZ$ with affine coarse space; see \cite{Ben13}. Moreover,  by   \cite{Javan3}, the stack $\mathcal{C}_{d;n,\QQ}$ is uniformisable, i.e., there is a smooth affine scheme  $U:=U_{d,n}$ over $\QQ$ and a finite \'etale morphism $U\to \mathcal{C}_{d;n,\QQ}$.   Now,  if $(d,n)\neq (3,2)$
the natural period map  on the smooth affine scheme   $U_{\CC}^{\an}$ is injective on tangent spaces by a theorem of Griffiths \cite{Fl86} (see also \cite{GriffithsInfTorelli}), as smooth hypersurfaces of degree $d$ in $\mathbb{P}^{n+1}$ satisfy the infinitesimal Torelli property (as $(d,n)\neq (3,2)$). This implies that the associated period map on $U_{\CC}^{\an}$  has finite fibres (see     \cite[Thm.~2.8]{JL}).   
 
The next result says that, for $k$ an algebraically closed field of characteristic zero,   the stack    $\mathcal{C}_{d;n, k}$ is geometrically hyperbolic over $k$. That is, the moduli space of pointed maps from any given pointed variety into the stack is finite, i.e., such maps   to $\mathcal{C}_{d;n,k}$ are \emph{rigid} and form a \emph{bounded} moduli space.
 
 \begin{theorem}  Let $k$ be an algebraically closed field of characteristic zero, and 
  let $X$ be a smooth hypersurface of degree $d$ in $\mathbb{P}^{n+1}_k$ over $k$.
Let $Y$ be an integral   variety over $k$ and let $y$ in $Y(k)$. Then, the set of $Y$-isomorphism classes of smooth hypersurfaces $\mathcal{X}$ of degree $d$ in $\mathbb{P}^{n+1}_Y$ such that $\mathcal{X}_y$ is isomorphic to $X$ over $k$ is finite.
 \end{theorem}
 \begin{proof} By a standard ``cyclic covering'' argument, we may and do assume that $(d,n)\neq (3,2)$. 
 Moreover, replacing $k$ by a field extension if necessary, we may and do assume that $k$ is uncountable.
Let $U$ be a smooth affine scheme over $k$ such that   there is a finite \'etale morphism $U\to \mathcal{C}_{d;n,k}$ of stacks (see \cite{Javan3}). Since $U$ admits a quasi-finite complex-analytic period map (up to Galois conjugation) by Griffiths's theorem (see for instance \cite{Fl86}), it follows from Theorem \ref{thm:geom_hyp_2} that $U$  is geometrically hyperbolic over $k$. Now, by a standard   descent argument (Remark \ref{remark:cw_geomhyp}), we deduce that the finite type separated Deligne-Mumford algebraic stack $\mathcal{C}_{d;n,k}$ is geometrically hyperbolic over $k$. In particular, as $k$ is uncountable, it follows from Remark \ref{remark:geomhyp} that, for every integral normal variety $Y$ over $k$, every point $y$ in $Y(k)$, and every $x$ in $\mathcal{C}_{d;n}(k)$, the set of morphisms $f:Y\to \mathcal{C}_{d;n,k}$ with $f(y)$ isomorphic to $x$ is finite. 
 \end{proof}

We now use the geometric hyperbolicity of the stack to show that its arithmetic hyperbolicity   persists over field extensions. Concerning arithmetically hyperbolic \emph{stacks},  we follow the conventions of  \cite[\S 4]{JLalg}.
 
 \begin{theorem}\label{thm:ci_arhyp} Let $L$ be an algebraically closed field of characteristic zero.
 The stack $\mathcal{C}_{d;n,\Qbar}$ is arithmetically hyperbolic over $\Qbar$ if and only if  $\mathcal{C}_{d;n,L}$ is arithmetically hyperbolic over $L$.
 \end{theorem}
 \begin{proof} 
Let $U:=U_{d,n}\to \mathcal{C}_{d;n,\Qbar}$ be a finite \'etale surjective morphism with $U$ a smooth affine   scheme over $\Qbar$ (see \cite{Javan3}). Since $\mathcal{C}_{d;n,\Qbar}$ is arithmetically hyperbolic over $\Qbar$ (by assumption) and $U\to \mathcal{C}_{d;n,\Qbar}$ is quasi-finite, it follows from \cite[Proposition~4.16]{JLalg}  that $U$ is arithmetically hyperbolic over $\Qbar$. Since $U$ admits a quasi-finite   period map (as explained above), it follows from Theorem \ref{thm:geom_hyp_2} that $U_L$ is geometrically hyperbolic over $L$. Thus, as $U$ is arithmetically hyperbolic over $k:=\Qbar$ and $U_L$ is geometrically hyperbolic over   $L$, we conclude that $U_L$ is arithmetically hyperbolic over $L$ (Proposition \ref{prop:gen0}). Now, as $U_L\to \mathcal{C}_{d;n,L}$ is finite \'etale, it follows from   stacky Chevalley-Weil   \cite[Theorem~5.1]{JLalg} that $\mathcal{C}_{d;n,L}$ is arithmetically hyperbolic over $L$.
 \end{proof}

\begin{lemma} \label{lem:finiteness} Let  $K$ be a number field and let $S$ be a finite set of finite places of $K$. Then, for every finite type affine group scheme $G$ over $\OO_{K,S}$,  the set of $\OO_{K,S}$-isomorphism classes of $G$-torsors over $\OO_{K,S}$ is finite.
\end{lemma}
\begin{proof} 
This is     \cite[Proposition~5.1]{GilleMoretBailly}.
\end{proof}
 
We now show that the Shafarevich conjecture for   hypersurfaces over number fields  \cite[Conjecture~1.4]{JL}    \emph{implies} a finiteness result for    hypersurfaces over finitely generated fields of characteristic zero.

 \begin{proof}[Proof of Theorem \ref{thm:hypsurf_intro}] The assumption is that, for every number field $K$ and every finite set $S$ of finite places of $K$, the set of $\OO_{K,S}$-isomorphism classes of smooth hypersurfaces of degree $d$ in $\mathbb{P}^{n+1}_{\OO_{K,S}}$ is finite. Now, an object of the groupoid   $\mathcal{C}_{d;n}(\OO_{K,S})$ is given by the data of a Brauer-Severi scheme $P$ over $\OO_{K,S}$ and a smooth hypersurface of degree $d$ in $P$. Therefore, 
by \cite[Lemma~4.8]{JL} and the finiteness of $\OO_{K,S}$-isomorphism classes of $(n+1)$-dimensional Brauer-Severi schemes over $\OO_{K,S}$ (which follows from Lemma \ref{lem:finiteness}), the assumption (in the statement of the theorem) implies that $\mathcal{C}_{d;n,\Qbar}$ is arithmetically hyperbolic over $\Qbar$.    Now, let $A$ be a $ \ZZ$-finitely generated normal integral domain of characteristic zero with fraction field $K$. Let $k:=\overline{K}$ be an algebraic closure of $K$ and note that $\mathcal{C}_{d;n,k}$ is arithmetically hyperbolic over $k$ by Theorem \ref{thm:ci_arhyp}. Therefore, since $A$ is normal and $\mathcal{C}_{d;n} $ is a finite type separated Deligne-Mumford algebraic stack over $\ZZ$, we conclude that the set of isomorphism classes of objects of $\mathcal{C}_{d;n}(A)$ is finite from the twisting lemma \cite[\S 4.3]{JLalg}; here we use that $A$ is integrally closed in its fraction field. In particular, the  set of    $A$-isomorphism classes of smooth hypersurfaces of degree $d$ in $\mathbb{P}^{n+1}_A$ is finite.      This concludes the proof. % linear isom and abs isom are the same as long as we avoid K3 surfaces.
 \end{proof}

\section{Non-density of integral points on the Hilbert scheme}\label{section:final}
 In this section  we prove Theorem  \ref{thm:non_density_hyps} (which is a statement about integral points on the Hilbert scheme) by using that the \emph{stack} of smooth hypersurfaces is weakly bounded and geometrically hyperbolic.
 
Recall that $\mathrm{Hilb}_{d,n}$ denotes the Hilbert scheme of smooth hypersurfaces of degree $d$ in $\mathbb{P}^{n+1}$ over $\ZZ$, and that this is a smooth affine scheme over $\ZZ$. As before, we write $\mathcal{C}_{d;n} = [\mathrm{PGL}_{n+2}\backslash \mathrm{Hilb}_{d,n}]$ for the stack of smooth hypersurfaces of degree $d$ in $\mathbb{P}^{n+1}$.
In our proof of Theorem \ref{thm:non_density_hyps} we  will  use Lemma \ref{lem:finiteness} to relate the non-density of  integral
 points on the Hilbert scheme to the non-density of integral points on the stack.

\begin{lemma} \label{lem:descent} Assume that, for every  number field $K$ and finite set of finite places $S$ of $K$, the set  $\mathrm{Hilb}_{d,n}(\OO_{K,S})$ is not dense in $\mathrm{Hilb}_{d,n}$.
Then, for every number field $K$ and finite set of finite places $S$ of $K$, the set of isomorphism classes of objects of $\mathcal{C}_{d;n}(\OO_{K,S})$ is not dense in $\mathcal{C}_{d;n}$.
\end{lemma}
\begin{proof}
Let $K$ be a number field and let $S$ be a finite set of finite places of $K$. Suppose that $\mathcal{C}_{d;n}(\OO_{K,S})$ is dense in $\mathcal{C}_{d;n}$. Note that the set of $\OO_{K,S}$-isomorphism classes of $\mathrm{PGL}_{n+2}$-torsors over $\OO_{K,S}$ is finite. Let $r\geq 1$  be an integer, and let $T_1,\ldots, T_r$ be representatives for all the $\PGL_{n+2}$-torsors over $\OO_{K,S}$, up to $\OO_{K,S}$-isomorphism. Let $L$ be a number field over $K$ and let $T$ be a finite set of finite places containing all the places of $L$ lying over $S$ such that, for every $i=1,\ldots, r$, the $\mathrm{PGL}_{n+2}$-torsor $T_i$ is trivial over $\OO_{L,T}$. For every $x$ in  $\mathcal{C}_{d;n}(\OO_{K,S})$, the fibre of the torsor $\mathrm{Hilb}_{d,n}\to \mathcal{C}_{d;n}$ over $x$ has a dense set of $\OO_{L,T}$-points. This implies that the set of $\OO_{L,T}$-points $\mathrm{Hilb}_{d,n}(\OO_{L,T})$ of $\mathrm{Hilb}_{d,n}$ is dense.
\end{proof}

\begin{lemma}[Non-density \`a la Chevalley-Weil]\label{lem:cw}
Let $X\to Y$ be a finite \'etale morphism of varieties over $k$. Assume that, for every $\ZZ$-finitely generated subring $A\subset k$ and every (finite type separated) model $\mathcal{X}$ for $X$ over $A$,  the set $\mathcal{X}(A)$ is not dense.
Then, for every $\ZZ$-finitely generated subring  $A\subset k$ and every model $\mathcal{Y}$ for $Y$ over $A$,  the set $\mathcal{Y}(A)$ is not dense in $Y$.
\end{lemma}
\begin{proof}
Use the descent argument used to prove Lemma \ref{lem:descent}  and the following well-known extension of Hermite's finiteness theorem:  if $D\geq 1$ is  an integer and $A$ is  a $\ZZ$-finitely generated normal integral domain of characteristic zero, then the set of $A$-isomorphism classes of finite \'etale morphisms $B\to \Spec A$ of degree at most $d$  is finite    \cite{smallness}.
\end{proof}

\begin{proof}[Proof of Theorem \ref{thm:non_density_hyps}]    
By \cite{Javan3},  there is an integer $D\geq 3$, a smooth affine scheme $U$ over $\ZZ[1/D]$ and a finite \'etale Galois morphism $U\to \mathcal{C}_{d;n, \ZZ[1/D]}$ of stacks over $\ZZ[1/D]$. The assumption on the non-density of integral points on the Hilbert scheme $\mathrm{Hilb}_{d,n}$ implies by  Lemma \ref{lem:descent} that, for every number field $K$ and every finite set of finite places $S$ of $K$,  we have that $\mathcal{C}_{d;n}(\OO_{K,S})$ is not dense in $\mathcal{C}_{d;n}$. Then, as $U\to \mathcal{C}_{d;n, \ZZ[1/D]}$ is surjective, we conclude that, for every number field $K$  and every finite set of finite places $S$ of $K$, the set $U(\OO_{K,S})$ is not dense  in $U$. Note that     $U_{\CC}$ is a smooth affine scheme over $\CC$ which admits a quasi-finite period map  (Section \ref{section:hyps}). Therefore, it follows from Theorem \ref{thm:criterion_nd} and Theorem \ref{thm:geom_hyp_2}  that the non-density of integral points on $U$     persists over finitely generated fields, i.e., for every $\ZZ$-finitely generated integral domain $A$ of characteristic zero, the set $U(A)$ is not dense in $U$. Let $H':= \mathrm{Hilb}_{d,n,\ZZ[1/D]} \times_{\mathcal{C}} U$, where $\mathcal{C}:=\mathcal{C}_{d;n,\ZZ[1/D]}$. Note that $H'\to H$ is  a finite \'etale morphism of schemes over $\ZZ[1/D]$. Since $H'\to U$ is surjective, it follows that, for every $\ZZ$-finitely generated integral domain $A$ of characteristic zero, the set $H'(A)$ is not dense in $H'$.  Therefore, by
 Lemma \ref{lem:cw}, as $H'\to \mathrm{Hilb}_{d,n,\ZZ[1/D]}$ is finite \'etale, we conclude that, for every $\ZZ$-finitely generated   integral domain $A$ of characteristic zero, the set   $\mathrm{Hilb}_{d,n}(A)$ is not dense.
\end{proof}

\begin{remark}
One may wonder whether arguing on the stack (or on $U$) is   necessary and whether one could   simply argue only on the Hilbert scheme in the proof of  Theorem \ref{thm:non_density_hyps}. The problem is that the Hilbert scheme  $\mathrm{Hilb}_{d,n,\CC}$ (over the complex numbers) does \textbf{not} admit a quasi-finite period map, is not geometrically hyperbolic, and is not weakly bounded.  Thus, to use our results on varieties with a quasi-finite period map, one is forced (in our proof of Theorem \ref{thm:non_density_hyps}) to argue   on the moduli stack of smooth hypersurfaces (or its finite \'etale atlas $U$).
\end{remark}

  \bibliography{refsperiod}{}
\bibliographystyle{alpha}

\end{document}